\documentclass[12pt,a4]{amsart}
\usepackage{amsmath}
\usepackage{amssymb}
\usepackage{amsthm}
\usepackage{enumitem}
\usepackage{url}
\usepackage{soul}

\usepackage{color}
\usepackage{dsfont}
\usepackage{epsfig}
\usepackage[hidelinks]{hyperref}
\usepackage{setspace}
\usepackage[authoryear,round,sort]{natbib}
\usepackage{comment}
\usepackage{booktabs}
\usepackage{float}
\usepackage{stmaryrd}
\numberwithin{equation}{section}

\def\R{\mathbb R}
\usepackage{geometry}
\theoremstyle{plain}

\newtheorem{theorem}{Theorem}[section]
\newtheorem{lemma}{Lemma}[section]

\theoremstyle{definition}

\theoremstyle{remark}
\newtheorem{rk}{Remark}[section]
\expandafter\let\expandafter\oldproof\csname\string\proof\endcsname
\let\oldendproof\endproof
\renewenvironment{proof}[1][\proofname]{%
  \oldproof[\noindent\textbf{#1.} ]%
}{\oldendproof}

\newcommand{\E}{\mathbb{E}}

\newcommand{\be}{\begin{equation}}
\newcommand{\ee}{\end{equation}}
\newcommand{\by}{\begin{eqnarray*}}
\newcommand{\ey}{\end{eqnarray*}}
\renewcommand{\le}{\leqslant}

\renewcommand{\leq}{\leqslant}
\renewcommand{\geq}{\geqslant}

\usepackage{xcolor}
\definecolor{dark-red}{rgb}{0.4,0.15,0.15}
\definecolor{dark-blue}{rgb}{0.15,0.15,0.4}
\definecolor{medium-blue}{rgb}{0,0,0.5}
\hypersetup{
    colorlinks, linkcolor={red},
    citecolor={blue}, urlcolor={blue}
}
\allowdisplaybreaks

\begin{document}
	\title[De Bruijn's identity for fBm]{Entropy flow and De Bruijn's identity for a class of stochastic differential equations driven by fractional Brownian motion}
	\author{Michael C.H. Choi, Chihoon Lee, and Jian Song}
	\address{Institute for Data and Decision Analytics, The Chinese University of Hong Kong, Shenzhen, Guangdong, 518172, P.R. China and Shenzhen Institute of Artificial Intelligence and Robotics for Society}
	\email{michaelchoi@cuhk.edu.cn}
	\address{School of Business, Stevens Institute of Technology, Hoboken, NJ, 07030, USA and Institute for Data and Decision Analytics, The Chinese University of Hong Kong, Shenzhen, Guangdong, 518172, P.R. China}
	\email{clee4@stevens.edu}
	\address{School of Mathematics, Shandong University, Jinan, Shandong, 250100, P.R. China}
	\email{txjsong@hotmail.com}
	\date{\today}
	\maketitle
	
	\begin{abstract}
		
		Motivated by the classical De Bruijn's identity for the additive Gaussian noise channel, in this paper we consider a generalized setting where the channel is modelled via stochastic differential equations driven by fractional Brownian motion with Hurst parameter $H\in(0,1)$. We derive generalized De Bruijn's identity for Shannon entropy and Kullback-Leibler divergence by means of It\^o's formula, and present two applications. In the first application we demonstrate its equivalence with Stein's identity for Gaussian distributions, while in the second application, we show that for $H \in (0,1/2]$, the entropy power is concave in time while for $H \in (1/2,1)$ it is convex in time when the initial distribution is Gaussian. Compared with the classical case of $H = 1/2$, the time parameter plays an interesting and significant role in the analysis of these quantities. 		\smallskip
		
		\noindent \textbf{AMS 2010 subject classifications}: 60G22, 60G15
		
		\noindent \textbf{Keywords}: fractional Brownian motion; De Bruijn's identity; Fokker-Planck equation; entropy power
	\end{abstract}
	
	
	
	\section{Introduction}
	
	Consider an additive Gaussian noise channel modelled by
	$$X_t = X_0 + \sqrt{t}Z,$$
	where $t \geq 0$, $Z$ is a standard normal random variable and the initial value $X_0$ is independent of $Z$. In information theory, the classical De Bruijn's identity, first studied by \cite{Stam59}, establishes a relationship between the time derivative of the entropy of $X_t$ to the Fisher information of $X_t$. While such a Gaussian channel is very popular in the literature (see e.g. \cite{GSV05,PV06}), in recent years researchers have been investigating into various generalizations of the noise channel. This includes the Fokker-Planck channel \cite{WJL17} in which it is modelled via stochastic differential equation driven by Brownian motion with general drift and diffusion, and also the dependent case \cite{KA16} where $X_0$ and $Z$ are jointly distributed as Archimedean or Gaussian copulas.
	
	In reality however, the channel may exhibit features that are not adequately modelled by the classical model. For example, in the area of Eternet traffic \cite{WTLW95}, it has been reported that the traffic exhibits self-similarity and long-range dependence. Similar phenomenon is also observed in analyzing video conference traffic \cite{BSTW95}. As these non-standard characteristics, in particular long-range dependence, cannot be effectively captured in the traditional additive Gaussian noise model, it motivates us to consider channel driven by fractional Brownian motion naturally as a possible generalization. In particular, long-range dependence is a significant feature possessed by fractional Brownian motion with Hurst parameter greater than $1/2$. In this paper we derive generalized De Bruijn's identity for such channel and discuss its relationship with Stein's identity as well as entropy power. Interestingly, the time paramter $t$ and the Hurst parameter $H$ of the fractional Brownian motion both play an important role in these results.
	
	Our mathematical contributions in this paper are as follows.
	
	\begin{itemize}
		
		\item We derive a generalized version of the celebrated De Bruijn's identity for channel driven by fractional Brownian motion. It involves a combination of techniques such as It\^o's formula and the Fokker-Planck equation.
		
		\item We build the connection and prove the equivalence between Stein's identity for Gaussian distributions and the generalized De Bruijn's identity, when the initial noise is a normal distribution.
		
		\item As another application of the generalized De Bruijn's identity, we demonstrate that the convexity/concavity of the entropy power depends on the Hurst parameter. This phenomenon is not observed in the classical Brownian motion case.
	\end{itemize}
	
	The rest of the paper is organized as follows. In Section \ref{sec:gendbi}, we first introduce the channel driven by fractional Brownian motion, followed by stating the results for generalized De Bruijn's identity as well as their proofs. In Section \ref{sec:applications}, we present two applications. In Section \ref{subsec:Stein}, we prove the equivalence between the generalized De Bruijn's identity and the Stein's identity for Gaussian distribution, while in Section \ref{subsec:entropypower}, we prove that the entropy power is convex or concave depending on the value of $H$. Finally in Section \ref{sec:conclusion}, we conclude our paper along with future research directions.
	
	Before we proceed to the main results of the paper, we first review a few important concepts that will be frequently used in subsequent sections. A \textbf{fractional Brownian motion (fBm)} $B^{H} = (B^{H}_t)_{t \geq 0}$ with Hurst parameter $H \in (0,1)$ is a centered Gaussian process with stationary increments and covariance function given by
	$$\E B^H_s B^H_t = \dfrac{1}{2} (t^{2H} + s^{2H} - |t-s|^{2H}).$$
	For further references of fBm, we refer readers to \cite{MV68}. The \textbf{Shannon entropy} of a random variable $X$ with probability density function $f_X$, denoted by $h(X)$, is given by
	\begin{align}\label{eq:entropy}
		h(X) := -\E(\ln f_X(X)) = - \int_{\mathbb R} f_X(x) \ln f_X(x)\, dx\,.
	\end{align}
	Let $b : \mathbb{R} \to (0,\infty)$ be a positive function. The \textbf{generalized Fisher information} with respect to $b$, first introduced by \cite{WJL17}, is given by
	\begin{align}\label{eq:gfi}
		J_b(X) := \E \left[b(X) \left(\dfrac{\partial}{\partial x} \ln f_X(X) \right)^2\right].
	\end{align}
	Note that when $b = 1$, $J_1(X)$ is simply the classical Fisher information. When $X$ follows a parametric distribution, say with location parameter $\theta$, then the Cram\'{e}r-Rao lower bound states that the variance of any unibased estimator of $\theta$ is lower bounded by the reciprocal of $J_1(X)$. The \textbf{Kullback–Leibler (KL) divergence}, or the relative entropy, between $X$ and random variable $Y$ with density $f_Y$, written as $K(X||Y)$, is given by
	\begin{align}\label{eq:kldiv}
		K(X||Y) := \E \left[ \ln \dfrac{f_X(X)}{f_Y(X)} \right] = \int_{\mathbb R} f_X(x) \ln \dfrac{f_X(x)}{f_Y(x)}\, dx.
	\end{align}
	For two random variables $X$ and $Y$, the \textbf{relative Fisher information} with respect to $b$ is
	\begin{align}\label{eq:relFisher}
		J_b(X||Y) := \E \left[b(X) \left(\dfrac{\partial}{\partial x} \ln \dfrac{f_X(X)}{f_Y(X)} \right)^2\right].
	\end{align}
\section{Generalized De Bruijn's identity}\label{sec:gendbi}
In this section, we derive the generalized De Bruijn's identity for channel modelled via stochastic differential equation driven by fractional Brownian motion (fBm). More precisely, consider a channel governed by
	\begin{equation}\label{sde}
	dX_t =\sigma(X_t)\circ dB_t^H,
	\end{equation}
with initial value $X_0=x_0,$ where $(B_t^H)_{t \geq 0}$ is a fBm with Hurst parameter $H \in (0,1)$.  The stochastic integral $\sigma(X_t)\circ dB_t^H$ is in the ``pathwise'' sense, i.e., if $H\in (1/2,1)$, the integral is understood as Young's integration (\cite{Y36}); if $H\in(1/4,1)$ it is understood in the rough paths sense of Lyons (\cite{CQ02}); if $H\in(1/6,1)$ it is understood in the sense of symmetric integral (\cite{RV93}); and if $H\in(0,1)$, and it is understood as $m$-order Newton-Cotes functional (\cite{GNRV05}). 

By Theorem 2.10 in \cite{N08}, assuming that the diffusion coefficient $\sigma(x)$ is sufficiently regular (say, infinitely differentiable with bounded derivatives of all orders), the one-dimensional SDE \eqref{sde} with $H\in (0,1)$ has a unique solution $X_t= \varphi(B_t^H)$, where $\varphi'(x)=\sigma(\varphi(x))$ with $\varphi(0)=x_0$, which can be obtained by using the Doss-Sussman transformation as in \cite{Suss78}. Note that  the solution $X_t$ is a function of $B_t^H$, rather than a functional of $(B_s^H)_{0\le s\le t}$.  This particular form allows  functions of $X_t$ to have a simple It\^o's formula (Lemma \ref{lem:Ito})  without involving Malliavian derivatives. As a consequence, the Fokker-Planck equation (Lemma \ref{lem:FPe}) can be derived, and  furthermore, a Feynman-Kac type formula  can be also obtained for a class of partial differential equations (See Corollary 26 and Example 28 in \cite{BD07}). Note that by Remark 27 in \cite{BD07}, this type of formulas only hold for the SDEs driven by fractional Brownian motion in the commutative case which is in the form of \eqref{sde} if the dimension is one.

\begin{theorem}[Generalized De Bruijn's identity for Shannon entropy of fBm]\label{thm:dbi}
	Consider the channel $X = (X_t)_{t \geq 0}$ modelled by equation \eqref{sde} with Hurst parameter $H \in (0,1)$ and initial value $X_0 = x_0$. Assume that the diffusion coefficient $\sigma(x)\in C^\infty(\R)$ has bounded derivatives of all orders. The entropy flow of $X$ is given by 
	\begin{align}\label{formula}
		\dfrac{d}{dt} h(X_t) = Ht^{2H-1} \Bigg\{& J_{\sigma^2}(X_t) -\E\left[\frac{\partial^2}{\partial x^2}\sigma^2(X_t)\right] +\E\left[\sigma''(X_t)\sigma(X_t)+(\sigma'(X_t))^2\right]\Bigg\},
	\end{align}
	where we recall that the generalized Fisher information $J_{\sigma^2}(X_t)$ is defined in \eqref{eq:gfi}.
\end{theorem}

\begin{rk}
	Note that when $H=1/2$, the fBm $W=B^H$ is a Brownian motion, and the Stratonovich equation \eqref{sde} becomes
	\[dX_t = \dfrac{\sigma(X_t)\sigma'(X_t)}{2} dt+\sigma(X_t) \diamond dW_t\]
	where the stochastic integral is in the It\^o sense. Then formula \eqref{formula} coincides with the classical De Bruijn's identity \cite{WJL17,KA16}. That is, when $H = 1/2$, \eqref{formula} becomes
	\begin{align*}
		\dfrac{d}{dt} h(X_t) = \dfrac{1}{2} \Bigg\{& J_{\sigma^2}(X_t) -\E\left[\frac{\partial^2}{\partial x^2}\sigma^2(X_t)\right] +\E\left[\sigma''(X_t)\sigma(X_t)+(\sigma'(X_t))^2\right]\Bigg\},
	\end{align*}
	which is the result in \cite[Theorem $5$]{WJL17} with drift $\dfrac{\sigma(x)\sigma'(x)}{2}$ and diffusion coefficient $\sigma(x)$. In particular, when $\sigma = 1$, we have
	$$\dfrac{d}{dt} h(X_t) = \dfrac{1}{2} J_1(X_t).$$
\end{rk}

\begin{theorem}[Generalized De Bruijn's identity for KL divergence of fBm]\label{thm:relativedbi}
	Consider the channel $X = (X_t)_{t \geq 0}$ (resp.~$Y = (Y_t)_{t \geq 0}$) modelled by equation \eqref{sde} with Hurst parameter $H \in (0,1)$ and initial value $X_0 = x_0$ (resp.~$Y_0 = y_0$). Assume that the diffusion coefficient $\sigma(x)\in C^\infty(\R)$ has bounded derivatives of all orders. The time derivative of the KL divergence between $X_t$ and $Y_t$ is given by 
	\begin{align}\label{eq:KLdiv}
	\dfrac{d}{dt} K(X_t || Y_t) = - Ht^{2H-1} J_{\sigma^2} (X_t || Y_t),
	\end{align}
	where we recall that the relative Fisher information $J_{\sigma^2} (X_t || Y_t)$ is defined in \eqref{eq:relFisher}. In particular, $K(X_t || Y_t)$ is non-increasing in $t$.
\end{theorem}

\begin{rk}
	Note that again when $H = 1/2$, \eqref{eq:KLdiv} becomes
	$$\dfrac{d}{dt} K(X_t || Y_t) = - \dfrac{1}{2} J_{\sigma^2} (X_t || Y_t),$$
	which is \cite[Theorem $6$]{WJL17}.
\end{rk}

In the first two main results above, we assume that the initial value is $X_0 = x_0$. In the following result, we assume that the channel is of the form
\begin{align}\label{eq:channel}
X_t = X_0 + B^H_t,
\end{align}
where the initial value $X_0$ is independent of the fBm.  We shall derive the generalized De Bruijn's identity via the classical version:

\begin{theorem}[Deriving the generalized De Bruijn's identity via the classical De Bruijn's identity]\label{thm:derivegeneralized}
	Consider the channel $X = (X_t)_{t \geq 0}$ modelled by equation \eqref{eq:channel} with Hurst parameter $H \in (0,1)$, initial value $X_0$ independent of the fBm and has a finite second moment. The entropy flow of $X$ is given by 
	\begin{align}\label{formula2}
	\dfrac{d}{dt} h(X_t) = Ht^{2H-1} J_{1}(X_t).
	\end{align}
	In particular, when $X_0$ is a Gaussian distribution with mean $0$ and variance $\sigma_0^2$, we then have
	$$\dfrac{d}{dt} h(X_t) = \dfrac{Ht^{2H-1} }{\sigma_0^2 + t^{2H}}.$$
\end{theorem}

\subsection{Proof of Theorem \ref{thm:dbi}, Theorem \ref{thm:relativedbi} and Theorem \ref{thm:derivegeneralized}}

We first present two lemmas that will be used in our proofs of Theorem \ref{thm:dbi} and Theorem \ref{thm:relativedbi}.

\begin{lemma}[It\^o's formula]\label{lem:Ito}
	Consider the channel $X = (X_t)_{t \geq 0}$ modelled by equation \eqref{sde} with Hurst parameter  $H \in (0,1)$, initial value $X_0 = x_0$ and twice differentiable diffusion coefficient $\sigma(x)$. Suppose that $f(t,x)$ is any twice differentiable function of two variables. Assume that the functions $\sigma(x)$ and $f(t,x)$ and their (partial) derivatives are at polynomial growth. Then
	\begin{align*}
		f(t,X_t) =& f(0, x_0) +\int_0^t f_s(s,X_s)ds+\int_0^t f_x(s,X_s)\sigma(X_s)\diamond dB_s^H\\
		&+H\int_0^t s^{2H-1} \Big(f_{xx}(s,X_s)\sigma(X_s)+ f_x(s, X_s)\sigma'(X_s)\Big)\sigma(X_s) ds.
	\end{align*}
\end{lemma}

\begin{proof}
Note that $X_t=\varphi(B_t^H)$ with $\varphi'(x)=\sigma(\varphi(x))$ and $\varphi(0)=x_0.$
By It\^o's formula in Section 8 of \cite{AMN01} for $H\in (1/4, 1)$ and in Corollary 4.8 of \cite{CN05} for $H\in(0, 1/2)$, noting that $\varphi''(x)=\sigma(\varphi(x))\varphi'(x)$, we have
\begin{align*}
&f(t,X_t)=f(t, \varphi(B_t^H))\\
=&f(0, x_0)+\int_0^t f_s(s,\varphi(B_s^H)) ds+\int_0^t f_x(s,\varphi(B_s^H)) \varphi'(B_s^H) \diamond dB_s^H\\
&+H\int_0^t s^{2H-1} \Big(f_{xx}(s,\varphi(B_s^H))(\varphi'(B_s^H))^2+f_x(s,\varphi(B_s^H))\varphi''(B_s^H) \Big)ds\\
=&f(0, x_0)+\int_0^t f_s(s,X_s) ds+\int_0^t f_x(s,X_s) \sigma(X_s) \diamond dB_s^H\\
&+H\int_0^t s^{2H-1} \Big(f_{xx}(s,X_s)\sigma(X_s)+f_x(s,X_s)\sigma'(X_s)\Big)\sigma(X_s)ds.
\end{align*}
\end{proof}

\begin{rk}[The condition imposed on $\sigma(x)$  in Theorems \ref{thm:dbi} and \ref{thm:relativedbi}]  
The proofs of Theorems \ref{thm:dbi} and \ref{thm:relativedbi} rely on the It\^o's formula given in Lemma \ref{lem:Ito} and the existence and uniqueness of the solution to \eqref{sde}. On one hand, in order to apply Lemma \ref{lem:Ito}, it is natural to assume that $\sigma(x)$ is twice differentiable with the derivatives growing at most polynomially fast.  On the other hand, more regularity condition on $\sigma(x)$ was  imposed to establish the existence and uniqueness of the solution to \eqref{sde} in  Theorem 2.10 of \cite{N08} for small Hurst parameter $H$. More precisely, for $H\in(\frac1{4m+2},1)$ with $m\in \mathbb N$, it is assumed that, $\sigma(x)$ belongs to $C^{4m+1}(\mathbb R)$ and is Lipschitz.  

In Theorems \ref{thm:dbi} and \ref{thm:relativedbi} for all $H\in (0,1)$, we simply assume that $\sigma(x)$ is smooth and all derivatives are bounded, which clearly satisfies the conditions in Lemma \ref{lem:Ito} above and Theorem 2.10 of \cite{N08}.
\end{rk}

\begin{lemma}[Fokker-Planck equation]\label{lem:FPe}
	Consider the channel $X = (X_t)_{t \geq 0}$ modelled by equation \eqref{sde} with Hurst parameter $H \in (0,1)$ and initial value $X_0 = x_0$. Assume that the diffusion coefficient $\sigma(x)\in C^\infty(\R)$ has bounded derivatives of all orders. Let $P_t(x)$ be the probability density function of $X_t$, then
	\begin{align*}
		\dfrac{\partial }{\partial t} P_t(x) = H t^{2H-1} \left(- \dfrac{\partial}{\partial x} \sigma'(x) \sigma(x) P_t(x) + \dfrac{\partial^2}{\partial x^2} \sigma^2(x) P_t(x)\right).
	\end{align*}
\end{lemma}

\begin{proof}
	Let $g(x)$ be a twice differentiable function, and we substitute $f(t,x) = g(x)$ in Lemma \ref{lem:Ito}. We arrive at
	\begin{align}\label{eq:FPe}
		\dfrac{d}{dt} \E[g(X_t)] &= Ht^{2H-1} \left(\E\left[g_{xx}(X_t) \sigma^2(X_t)\right] + \E \left[g_x(X_t) \sigma'(X_t) \sigma(X_t)\right]\right).
	\end{align}
	Note that the left hand side of \eqref{eq:FPe} is
	$$\dfrac{d}{dt} \E[g(X_t)] = \int_{\mathbb R} g(x) \dfrac{\partial}{\partial t} P_t(x) \,dx.$$
	Using integration by part, the right hand side of \eqref{eq:FPe} can be written as
	$$Ht^{2H-1} \left( \int g(x) \left(\dfrac{\partial^2}{\partial x^2} \sigma^2(x) P_t(x)\right)\, dx - \int_{\mathbb R} g(x) \left(\dfrac{\partial}{\partial x} \sigma'(x) \sigma(x) P_t(x)\right)\,dx \right).$$
	The desired result follows since $g$ is arbitrary.
\end{proof}

\subsubsection{Proof of Theorem \ref{thm:dbi}}
Denote by $P_t(x)$ the probability density of $X_t$, and let $f(s,x)=-\ln P_s(x)$ in Lemma \ref{lem:Ito}. Then we have
\begin{align*}
&f_x(s,x)=-(P_s(x))^{-1}\frac{\partial}{\partial x}P_s(x),
\end{align*}
and
\begin{align*}
&f_{xx}(s,x)=(P_s(x))^{-2}\left(\frac{\partial}{\partial x}P_s(x)\right)^2-(P_s(x))^{-1}\frac{\partial^2}{\partial x^2}P_s(x).
\end{align*}
Thus,
\begin{align*}
&\E[f_x(s,X_s) \sigma'(X_s)\sigma(X_s)]=-\int_{\R}\frac{\partial}{\partial x}P_s(x) \sigma'(x)\sigma(x)dx\\
&=\int_{\R} P_s(x) (\sigma'(x)\sigma(x))'dx=\E\left[\sigma''(X_s)\sigma(X_s)+(\sigma'(X_s))^2\right],
\end{align*}
and 
\begin{align*}
&\E[f_{xx}(s,X_s) \sigma^2(X_s)]=\int_\R \left[ (P_s(x))^{-1}\left(\frac{\partial}{\partial x}P_s(x)\right)^2 -\frac{\partial^2}{\partial x^2}P_s(x) \right]\sigma^2(x)dx\\
&=\E\left[\sigma^2(X_s) \left(\frac\partial{\partial x}\ln P_s(X_s)\right)^2 \right]-\E\left[\frac{\partial^2}{\partial x^2}\sigma^2(X_s)\right].
\end{align*}
Therefore, we have the following formula.
\begin{align}
-\frac d{dt} \E[\ln P_t(X_t)]= Ht^{2H-1} \Bigg\{& \E\left[\sigma^2(X_t) \left(\frac\partial{\partial x}\ln P_t(X_t)\right)^2 \right]-\E\left[\frac{\partial^2}{\partial x^2}\sigma^2(X_t)\right]\\
&+\E\left[\sigma''(X_t)\sigma(X_t)+(\sigma'(X_t))^2\right]
\Bigg\}.\notag
\end{align}

\subsubsection{Proof of Theorem \ref{thm:relativedbi}}
In this proof, we write $P_t(x)$ to be the probability density of $X_t$, $Q_t(y)$ to be the probability density of $Y_t$ and let $f(s,x)=\ln \dfrac{P_s(x)}{Q_s(x)}$ in Lemma \ref{lem:Ito}. Then we have
\begin{align*}
	f_x(s,x) &= \left(\dfrac{P_s(x)}{Q_s(x)}\right)^{-1} \dfrac{\partial}{\partial x} \dfrac{P_s(x)}{Q_s(x)}, \\
	f_{xx}(s,x) &= \left(\dfrac{P_s(x)}{Q_s(x)}\right)^{-1} \left(\dfrac{\partial^2}{\partial x^2} \dfrac{P_s(x)}{Q_s(x)}\right) - \left(\dfrac{P_s(x)}{Q_s(x)}\right)^{-2} \left(\dfrac{\partial}{\partial x} \dfrac{P_s(x)}{Q_s(x)}\right)^2 \\
	&= \left(\dfrac{P_s(x)}{Q_s(x)}\right)^{-1} \left(\dfrac{\partial^2}{\partial x^2} \dfrac{P_s(x)}{Q_s(x)} \right) - \left(\dfrac{\partial}{\partial x} \ln \dfrac{P_s(x)}{Q_s(x)}\right)^2, \\
	f_s(s,x) &= \dfrac{1}{P_s(x)} \dfrac{\partial}{\partial s} P_s(x) - \dfrac{1}{Q_s(x)} \dfrac{\partial}{\partial s} Q_s(x).
\end{align*}
As a result, using integration by part we arrive at 
\begin{align*}
	\E [f_s(s,X_s)] &= - \int \left(\dfrac{\partial}{\partial s} Q_s(x) \right) \dfrac{P_s(x)}{Q_s(x)}\,dx, \\
	\E [f_x(s,X_s)\sigma'(X_s)\sigma(X_s)] &= - \int \left(\dfrac{\partial}{\partial x} \sigma'(x) \sigma(x) Q_s(x) \right) \dfrac{P_s(x)}{Q_s(x)}\,dx, \\
	\E [f_{xx}(s,X_s)\sigma^2(X_s)] &= \int \left(\dfrac{\partial^2}{\partial x^2} \sigma^2(x) Q_s(x) \right) \dfrac{P_s(x)}{Q_s(x)}\,dx - J_{\sigma^2}(X_s || Y_s). 
\end{align*}
Now, by Lemma \ref{lem:Ito} we note that
\begin{align*}
\dfrac{d}{dt} K(X_t || Y_t) &= \dfrac{d}{dt} \E \left[\ln \dfrac{P_t(X_t)}{Q_t(X_t)} \right] \\
							&= \E [f_t(t,X_t)] + Ht^{2H-1} \left(\E [f_x(t,X_t)\sigma'(X_t)\sigma(X_t)] + \E [f_{xx}(t,X_t)\sigma^2(X_t)]\right) \\
							&= - Ht^{2H-1}J_{\sigma^2} (X_t || Y_t),
\end{align*}
where the last equality follows from Lemma \ref{lem:FPe}. 

\subsubsection{Proof of Theorem \ref{thm:derivegeneralized}}

Let $t\geq0$ be such that $s^{H} = \sqrt{t}$ and denote $Z$ to follow the standard normal distribution. Using chain rule we have
\begin{align*}
\dfrac{d}{ds} h(X_s) &= \dfrac{d}{ds} h(X_0 + s^H Z) \\
&= \dfrac{d}{ds} h(X_0 + \sqrt{t} Z) \\
&= \dfrac{d}{dt} h(X_0 + \sqrt{t} Z) \dfrac{dt}{ds} \\
&= \dfrac{1}{2} J_1(X_0 + \sqrt{t} Z) 2H s^{2H-1} \\
&= Hs^{2H-1} J_{1}(X_s), 
\end{align*}
where the fourth equality follows from the classical De Bruijn's identity (see e.g. \cite{CJ06}). In particular when $X_0$ is Gaussian, then $X_s$ is also Gaussian with mean $0$ and variance $\sigma_0^2 + s^{2H}$. Since for normal distribution the Fisher information is the reciprocal of the variance, we have
\begin{align*}
\dfrac{d}{ds} h(X_s) = Hs^{2H-1} J_{1}(X_s) = \dfrac{Hs^{2H-1}}{\mathrm{Var}(X_s)} = \dfrac{Hs^{2H-1}}{\sigma_0^2 + s^{2H}}.
\end{align*}

\section{Applications}\label{sec:applications}

In this section, we present two applications of the generalized De Bruijn's identity. In the first application in Section \ref{subsec:Stein}, we demonstrate its equivalence with the Stein's identity for Gaussian distribution, while in Section \ref{subsec:entropypower}, we prove the convexity or the concavity of entropy power, which depends on the Hurst parameter $H$. Throughout this section, we assume that the channel is of the form
\begin{align*}
 	X_t = X_0 + B^H_t,
\end{align*}
where the initial value $X_0$ is independent of the fBm and the Hurst parameter $H \in (0,1)$.

\subsection{Equivalence of the generalized De Bruijn's identity and Stein's identity for normal distribution}\label{subsec:Stein}

It is known that the classical De Bruijn's identity is equvialent to the Stein's identity for normal distribution as well as the heat equation identity, provided that the initial noise $X_0$ is Gaussian, see e.g. \cite{BDHS06,PSQ12}. These identities are equivalent in the sense that one can derive the others using any one of them. It is therefore natural for us to guess that the same equivalence also holds for the proposed generalized De Bruijn's identity. To this end, let us recall the classical Stein's identity for normal distribution. Writing $Y$ to be the normal distribution with mean $\mu$ and variance $\sigma^2$, the \textbf{Stein's identity} is given by
\begin{align}\label{eq:Stein}
	\E[r(Y)(Y - \mu)] = \sigma^2 \E \left[\dfrac{d}{dy}r(Y)\right],
\end{align}
where $r$ is a differentiable function such that the above expectations exist. In the following result, we prove that the generalized De Bruijn's identity presented in Theorem \ref{thm:derivegeneralized} is equivalent to the Stein's identity, 

\begin{theorem}[Equivalence of the generalized De Bruijn's identity \eqref{formula2} and Stein's identity]\label{thm:eqdbistein}
	Consider the channel $X = (X_t)_{t \geq 0}$ modelled by equation \eqref{eq:channel} with Hurst parameter $H \in (0,1)$ and initial Gaussian $X_0$ independent of the fBm. Then the generalized De Bruijn's identity \eqref{formula2} is equivalent to the Stein's identity \eqref{eq:Stein}.
\end{theorem}

\begin{rk}[Equivalence of the generalized De Bruijn's identity \eqref{formula} and Stein's identity]
While Theorem \ref{thm:eqdbistein} is established for channel \eqref{eq:channel}, it is natural to consider if the same equivalence holds between \eqref{formula} and Stein's identity for channel \eqref{sde}. One direction is straightforward: with \eqref{formula} we have the classical De Bruijn's identity by taking the diffusion coefficent to be $\sigma(x) = 1$, and so we can derive the Stein's identity. However, for the opposite direction, we did not manage to prove \eqref{formula} via the classical De Bruijn's identity in the presence of general diffusion coefficient $\sigma(x)$. The trick employed in proving Theorem \ref{thm:derivegeneralized}, where the diffusion coefficient is simply $1$, does not seem to carry over to this setting. If this can be proved by other means, then the equivalence could be established.
\end{rk}

\begin{proof}
	If we have the Stein's identity, then we can derive the classical De Bruijn's identity \cite{PSQ12}, and so we have the generalized De Bruijn's identity by Theorem \ref{thm:derivegeneralized}. For the other direction, if we have the generalized De Bruijn's identity, then we can derive the classical De Bruijn's identity by taking $H = 1/2$, and from it we can derive the Stein's identity by \cite{PSQ12}.
\end{proof}

\subsection{Convexity/Concavity of the entropy power}\label{subsec:entropypower}

Recall that the \textbf{entropy power} of a random variable $X$ is defined to be
\begin{align}\label{def:entropypower}
	N(X) := \dfrac{1}{2\pi e} e^{2h(X)}.
\end{align}
In the classical setting when the channel $X_t$ is of the form \eqref{eq:channel} with $X_0$ being an arbitrary initial noise, \cite{Costa85,Dembo89} prove that the entropy power of $N(X_t)$ is concave in time $t$. Recently in \cite{KA16} the authors extend the concavity of entropy power to the dependent case where the dependency structure between the initial value $X_0$ and the channel $X_t$ is specified by Archimedean and Gaussian copulas. In our case, interestingly convexity/concavity of the entropy power depends on the Hurst parameter $H$:

\begin{theorem}[Convexity/Concavity of the entropy power]\label{thm:convexconcave}
	Consider the channel $X = (X_t)_{t \geq 0}$ modelled by equation \eqref{eq:channel} with Hurst parameter $H \in (0,1)$, initial value $X_0$ independent of the fBm and has a finite second moment. We have 
	\begin{align*}
		\dfrac{d^2}{d t^2} N(X_t) &= 2 N(X_t) \left(2H^2 t^{4H-2} J_1(X_t)^2 + H(2H-1)t^{2H-2}J_1(X_t) + Ht^{2H-1} \dfrac{d}{d t} J_1(X_t)\right) \\
		&= 2 N(X_t) g(t,H,X_t),
	\end{align*}
	where $g(t,H,X_t) := 2H^2 t^{4H-2} J_1(X_t)^2 + H(2H-1)t^{2H-2}J_1(X_t) + Ht^{2H-1} \dfrac{d}{d t} J_1(X_t)$. Consequently,
	\begin{align*}
		N(X_t) &=\begin{cases} \text{convex in t} &\mbox{if } g(t,H,X_t) > 0, \\ 
		\text{concave in t} & \mbox{if } g(t,H,X_t) \leq 0. \end{cases}
	\end{align*}
	In particular, when $X_0$ is a Gaussian distribution with mean $0$ and variance $\sigma_0^2$, we then have $g(t,H,X_t) = H(2H-1)t^{2H-2}J_1(X_t)$ and
	\begin{align*}
	N(X_t) &=\begin{cases} \text{convex in t} &\mbox{if } H \in (1/2,1), \\ 
	\text{concave in t} & \mbox{if } H \in (0,1/2]. \end{cases}
	\end{align*}
\end{theorem}

\begin{rk}
	In the special case when $H = 1/2$ and $X_0$ is Gaussian, we retrieve the classical result that $N(X_t)$ is linear and hence concave (or convex) in $t$.
\end{rk}

\begin{rk}[The role of time parameter $t$]
	In Theorem \ref{thm:convexconcave}, we see that there are factors such as $t^{2H-1}$ or $t^{4H-2}$ appearing in the function $g(t,H,X_t)$. While these terms equal to $1$ in the classical $H = 1/2$ case, as we shall see in the proof these terms play an important and interesting role in determining the second-order behaviour of the entropy power in the general fBm case. Note that these terms all come from the factor $t^{2H-1}$ in front of the Fisher information $J_1$ in \eqref{formula2}.
\end{rk}

\begin{rk}[On establishing the convexity/concavity of the entropy power in model \eqref{sde}]
	While Theorem \ref{thm:convexconcave} is stated for model \eqref{eq:channel}, we can in fact state a similar result for model \eqref{sde} using Theorem \ref{thm:dbi}. However, there will not be a clear cut distinction between the two cases $H \leq 1/2$ and $H > 1/2$ as in Theorem \ref{thm:convexconcave}; it will also depend on the derivatives of the diffusion coefficient $\sigma(x)$, which may not be tractable in general.
\end{rk}

\begin{proof}
	Using the definition of the entropy power \eqref{def:entropypower}, we have	\begin{align*}
		\dfrac{d^2}{d t^2} N(X_t) &= {\color{blue} \dfrac{d}{dt}\left(2 N(X_t) \dfrac{d}{dt} h(X_t)\right)} \\
											&= 2N(X_t) \left( 2 \left(\dfrac{d}{dt} h(X_t)\right)^2 + \dfrac{d^2}{dt^2} h(X_t) \right) \\
											&= 2 N(X_t) \left(2H^2 t^{4H-2} J_1(X_t)^2 + H(2H-1)t^{2H-2}J_1(X_t) + Ht^{2H-1} \dfrac{d}{d t} J_1(X_t)\right) \\
											&= 2 N(X_t) g(t,H,X_t), 
	\end{align*}
	where we make use of the generalized De Bruijn's identity \eqref{formula2} in the third equality. Since $N(X_t) \geq 0$, convexity/concavity of $N(X_t)$ thus depends on the sign of the function $g(t,H,X_t)$. In particular, when $X_0$ is Gaussian with mean $0$ and variance $\sigma_0^2$, we have 
	\begin{align*}
		J_1(X_t) &= \dfrac{1}{\sigma_0^2 + t^{2H}}, \\
		\dfrac{d}{dt} J_1(X_t) &= \dfrac{- 2H t^{2H-1}}{(\sigma_0^2 + t^{2H})^2} = - 2H t^{2H-1} J_1(X_t)^2, \\
		g(t,H,X_t) &= H(2H-1) t^{2H-2} J_1(X_t) \\
				   &= \begin{cases} > 0 &\mbox{if } H \in (1/2,1), \\ 
				   \leq 0 & \mbox{if } H \in (0,1/2]. \end{cases}
	\end{align*}
\end{proof} 

Theorem \ref{thm:convexconcave} implies that, for channel of the form \eqref{eq:channel} and Gaussian distributed $X_0$, for $t \in [0,1]$, we have 
\begin{align*}
N(X_t) &\begin{cases} \leq t N(X_0) + (1-t) N(X_1) &\mbox{if } H \in (1/2,1), \\ 
\geq t N(X_0) + (1-t) N(X_1) & \mbox{if } H \in (0,1/2]. \end{cases}
\end{align*}
One application of the above equation lies in determining the so-called capacity region in a Gaussian interference channel; see the paper \cite{Costa85b}. It turns out that the concavity of entropy power is a crucial step in the proof of Theorem 2 in \cite{Costa85b}. With our Theorem \ref{thm:convexconcave}, it seems possible to study the capacity region in an interference channel with fBm noise and generalize the result to $H \in (0,1/2]$. We leave this as one of our future research directions.

\section{Conclusion}\label{sec:conclusion}

In this paper, we present the generalized De Bruijn's identity for the channel driven by fractional Brownian motion with Hurst parameter $H\in(0,1)$. Compared with the classical Brownian motion, i.e., $H=\frac12$, in our setting, the term $t^{2H-1}$ in general does not degenerate unless $H=\frac12$, and thus plays an essential role in the derivation of the identity. Consequently, we also investigate its equivalence with the Stein's identity and study the second-order behaviour of the entropy power. We hope that this paper can open new doors in analyzing the De Bruijn's identity in a more general context to model phenomenon such as self-similarity and long-range dependency.

There are at least two future research directions.  First, as mentioned in Section \ref{subsec:entropypower}, we can study the capacity region in an interference channel with fBm noise, where our convexity/concavity result should be an important step in the analysis. Second, we can consider more general channel driven by stochastic processes such as the Gaussian Volterra process \cite{V17}. Members of this broad family include fBm as well as the Riemann-Liouville process. This shall provide a unified framework in studying the De Bruijn's identity and its applications for channels driven by the more general Gaussian processes.\\

\noindent \textbf{Acknowledgements}.
The authors would like to thank the associate editor and the referee for constructive comments that improve the quality of the manuscript. Michael Choi acknowledges the support from the Chinese University of Hong Kong, Shenzhen grant PF01001143 and AIRS - Shenzhen Institute of Artificial Intelligence and Robotics for Society Project 2019-INT002.

\bibliographystyle{abbrvnat}
\bibliography{thesis}

\end{document}